\documentclass[11pt]{amsart}
\usepackage{amsmath}
\usepackage{amssymb}
\usepackage{amscd}

\def\NZQ{\mathbb}               
\def\NN{{\NZQ N}}
\def\QQ{{\NZQ Q}}
\def\ZZ{{\NZQ Z}}
\def\RR{{\NZQ R}}

\def\PP{{\NZQ P}}

\newtheorem{Theorem}{Theorem}[section]
\newtheorem{Lemma}[Theorem]{Lemma}
\newtheorem{Corollary}[Theorem]{Corollary}
\newtheorem{Proposition}[Theorem]{Proposition}
\newtheorem{Remark}[Theorem]{Remark}

\newtheorem{Example}[Theorem]{Example}

\let\epsilon\varepsilon
\let\phi=\varphi
\let\kappa=\varkappa

\begin{document}

\title{Analytic spread and associated primes of divisorial filtrations}
\author{Steven Dale Cutkosky}
\thanks{Partially supported by NSF grant DMS-2054394}

\address{Steven Dale Cutkosky, Department of Mathematics,
University of Missouri, Columbia, MO 65211, USA}
\email{cutkoskys@missouri.edu}


\begin{abstract} 
We prove that a classical theorem of McAdam about the  analytic spread of an ideal in a Noetherian local ring is true for divisorial filtrations on an  excellent local domain $R$ which is either of equicharacteristic zero or of dimension $\le 3$. In fact, the proof is valid whenever resolution of singularities holds.
\end{abstract}

\maketitle

\section{Introduction} Expositions of the theory of complete ideals, integral closure of ideals and their relation to valuation ideals, Rees valuations, analytic spread and birational morphisms can be found, from different perspectives,  in  \cite{ZS2}, \cite{HS}, \cite{Li2} and \cite{Li3}. The book \cite{HS} and the article \cite{Li3} contain references to  original work in this subject.  A survey of recent work on symbolic algebras is given in \cite{DDGHN}. A different notion of analytic spread for families of ideals is given in \cite{DM}. A recent paper exploring ideal theory in  2 dimensional normal local domains using geometric methods is \cite{ORWY}.

Let $R$ be a Noetherian local ring with maximal ideal $m_R$. If $I$ is an ideal in $R$, then the Rees algebra of $I$ is the graded $R$-algebra
$R[It]=\sum_{n\ge 0}I^nt^n$ and the analytic spread of $I$ is 
$$
\ell(I)=\dim R[It]/m_RR[It].
$$ 
It is natural to extend this definition to arbitrary filtrations of $R$.
Let $\mathcal I=\{I_n\}$ be a filtration on  $R$. The Rees algebra of the filtration is $R[\mathcal I]=\oplus_{n\ge 0}I_n$ and the analytic spread of the filtration $\mathcal  I$ is defined to be
\begin{equation}\label{eqN10}
\ell(\mathcal I)=\dim R[\mathcal I]/m_RR[\mathcal I].
\end{equation}
Thus if $I$ is an ideal in $R$ and $\mathcal I=\{I^n\}$ is the $I$-adic filtration, then $\ell(I)=\ell(\mathcal I)$.
The essential new phenomena in the case of arbitrary filtrations $\mathcal I$ is that the Rees algebra $R[\mathcal I]$ may not be Noetherian.  

An essential property of the analytic spread $\ell(I)$ of an ideal is the inequality (\cite[Proposition 5.1.6 and Corollary 8.3.9]{HS}) 
\begin{equation}\label{eqI1}
{\rm ht}(I)\le \ell(I)\le \dim R.
\end{equation}

If $\mathcal I$ is a filtration, then 
${\rm ht}(I_n)={\rm ht}(I_1)$ for all $n$ (\cite[equation (7)]{CPS}) so it is natural to define ${\rm ht}(\mathcal I)={\rm ht}(I_1)$.

We always have (\cite[Lemma 3.6]{CPS}) that 
\begin{equation}\label{eqZ1}
\ell(\mathcal I)\le \dim R
\end{equation}
so the second inequality of (\ref{eqI1}) always holds for filtrations. 
 However, the first inequality of (\ref{eqI1}),
$\mbox{ht}(\mathcal I)\le \ell(\mathcal I)$,
fails spectacularly, even attaining the condition that $\ell(\mathcal I)=0$ (\cite[Example 1.2, Example 6.1 and Example 6.6]{CPS}). The last two of these examples are of symbolic algebras of space curves, which are  divisorial filtrations. 

 The condition that  a filtration has  analytic spread zero has a simple ideal theoretic interpretation (\cite[Lemma 3.8]{CPS}).
 Suppose that  $\mathcal I=\{I_n\}$ is a filtration in a local ring $R$. Then the analytic spread
$\ell(\mathcal I)=0$ if and only if
$$
\mbox{For all $n>0$ and $f\in I_n$, there exists $m>0$ such that $f^m\in m_RI_{mn}$.}
$$

Let $I$ be an ideal in a Noetherian local domain $R$.
The integral closure of the Rees algebra $R[It]=\sum_{n\ge 0}I^nt^n$ in $R[t]$ is $\sum_{n\ge 0}\overline{I^n}t^n$. The ideal $\overline{I^n}$ in $R$ is called the integral closure of $I^n$ in $R$. Divisorial valuations of $R$ are defined later in this introduction.
There exist divisorial valuations $\nu_1,\ldots,\nu_r$ of $R$ called the Rees valuations of $I$ and $a_1\ldots,a_r\in \ZZ_{>0}$ such that 
$$
\overline{I^n}=I(\nu_1)_{na_1}\cap\cdots\cap I(\nu_r)_{na_r}
$$
for $n\in \NN$. This expression is irredundant, in the sense that for all $j$, there exists an $n$ such that  $\overline{I^n}\ne \cap_{n\ne j}I(\nu_i)_{na_i}$. This is proven (even for a generalization of this statement to arbitrary Noetherian rings) in \cite{Re} and \cite[Theorem 10.2.2]{HS}. 

We have the following remarkable theorem.

\begin{Theorem}\label{TheoremC4}(\cite{McA}, \cite[Theorem 5.4.6]{HS}) Let $R$ be a formally equidimensional local ring and $I$ be an ideal in $R$. Then $m_R\in \mbox{Ass}(R/\overline{I^n})$ for some $n$ if and only if $\ell(I)=\dim(R)$.
\end{Theorem}

The assumption of being formally equidimensional is not required for the if direction of Theorem \ref{TheoremC4} (this is Burch's theorem, \cite{Bu}, \cite[Proposition 5.4.7]{HS}).

We now consider divisorial filtrations. 
Given an ideal $I$ in $R$, the filtration $\{\overline{I^n}\}$ is an example of a divisorial filtration of $R$. The filtration 
$\{\overline{I^n}\}$ is Noetherian if $R$ is universally Nagata. The symbolic powers of an ideal is another example of a divisorial filtration.

 Let $R$ be a  Noetherian local domain of dimension $d$ with quotient field $K$.  Let $\nu$ be a discrete valuation of $K$ with valuation ring $V_{\nu}$ and maximal ideal $m_{\nu}$.  Suppose that $R\subset V_{\nu}$. Then for $n\in \NN$, define valuation ideals
$$
I(\nu)_n=\{f\in R\mid \nu(f)\ge n\}=m_{\nu}^n\cap R.
$$
 
  A divisorial valuation of $R$ (\cite[Definition 9.3.1]{HS}) is a valuation $\nu$ of $K$ such that if $V_{\nu}$ is the valuation ring of $\nu$ with maximal ideal $m_{\nu}$, then $R\subset V_{\nu}$ and if $p=m_{\nu}\cap R$ then $\mbox{trdeg}_{\kappa(p)}\kappa(\nu)={\rm ht}(p)-1$, where $\kappa(p)$ is the residue field of $R_{p}$ and $\kappa(\nu)$ is the residue field of $V_{\nu}$.   If $\nu$ is divisorial valuation of $R$ such that $m_R= m_{\nu}\cap R$, then $\nu$ is called an $m_R$-valuation.
   
 By \cite[Theorem 9.3.2]{HS}, the valuation ring of every divisorial valuation $\nu$ is Noetherian, hence is a  discrete valuation. 
 Suppose that  $R$ is an excellent local domain. Then a valuation $\nu$ of the quotient field $K$ of $R$ which is nonnegative on $R$ is a divisorial valuation of $R$ if and only if the valuation ring $V_{\nu}$  of $\nu$ is essentially of finite type over $R$ (\cite[Lemma 5.1]{CPS1}).
 
 In general, the filtration $\mathcal I(\nu)=\{I(\nu)_n\}$ is not Noetherian; that is, the graded $R$-algebra $\sum_{n\ge 0}I(\nu)_nt^n$ is not a finitely generated $R$-algebra. In a two dimensional normal local ring $R$, the condition that the filtration of valuation ideals $\mathcal I(\nu)$ is Noetherian for all $m_R$-valuations $\nu$ dominating $R$  is the condition (N) of Muhly and Sakuma \cite{MS}. It is proven in \cite{C2} that a complete normal local ring of dimension two satisfies condition (N) if and only if its divisor class group is a torsion group.
 
 An integral  divisorial filtration of $R$ (which we will also refer to as a divisorial filtration) is a filtration $\mathcal I=\{I_m\}$ such that  there exist divisorial valuations $\nu_1,\ldots,\nu_s$ and $a_1,\ldots,a_s\in \ZZ_{\ge 0}$ such that for all $m\in \NN$,
$$
I_m=I(\nu_1)_{ma_1}\cap\cdots\cap I(\nu_s)_{ma_s}.
$$

$\mathcal I$ is called an $\RR$-divisorial filtration if $a_1,\ldots,a_s\in \RR_{>0}$ and $\mathcal I$ is called a $\QQ$-divisorial filtration if $a_1,\ldots,a_s\in \QQ$. If $a_i\in \RR_{>0}$, then
$$
I(\nu_i)_{na_i}:=\{f\in R\mid \nu_i(f)\ge na_i\}=I(\nu_i)_{\lceil na_i\rceil},
$$
where $\lceil x\rceil$ is the round up of a real number.

It is shown in \cite[Theorem 4.6]{CPS} that 
 the ``if" statement of  Theorem \ref{TheoremC4} is true for $\RR$-divisorial filtrations of a local domain $R$. This theorem is stated for divisorial filtrations, but the proof is valid for $\RR$-divisorial filtrations. 
 
 \begin{Theorem}(\cite[Theorem 4.6]{CPS})\label{TheoremN2} Suppose that $R$ is a Noetherian local domain and $\mathcal I=\{I_n\}$ is  an $\RR$-divisorial filtration on $R$ such that $\ell(\mathcal I)=\dim R$. Then there exists $n_0\in \NN$ such that $m_R\in \mbox{Ass}(R/\overline{I^n})$ for all $n\ge n_0$.
 \end{Theorem} 
 
 An interesting question is if the converse of Theorem \ref{TheoremC4} is also true for divisorial filtrations of a local ring $R$.
 We prove this for  excellent normal local rings for which resolution of singularities holds in this paper. 
  
\begin{Theorem}\label{Theorem1} Let $R$ be an  excellent local ring of equicharacteristic 0, or  of dimension $\le 3$. Let $\mathcal I=\{I_n\}$ be a $\QQ$-divisorial filtration on $R$. Suppose that $m_R\in \mbox{Ass}(R/I_{m_0})$ for some $m_0\in \ZZ_{>0}$. Then the analytic spread of $\mathcal I$ is $\ell(\mathcal I)=\dim R$. Further, there exists $n_0\in \ZZ_{>0}$ such that $m_R\in \mbox{Ass}(R/I_n)$ if $n\ge n_0$.
\end{Theorem}

The following theorem is immediate from Theorems \ref{TheoremN2} and \ref{Theorem1}. Theorem \ref{Theorem2}  generalizes
Theorem \ref{TheoremC4} from ideals to filtrations.  

\begin{Theorem}\label{Theorem2} Let $R$ be an excellent local ring of equicharacteristic 0, or  of dimension $\le 3$. Let $\mathcal I=\{I_n\}$ be a $\QQ$-divisorial filtration on $R$. Then the following are equivalent.
\begin{enumerate}
\item[1)] The analytic spread of $\mathcal I$ is $\ell(\mathcal I)=\dim R$.
\item[2)] There exists $n_0\in \ZZ_{>0}$ such that $m_R\in \mbox{Ass}(R/I_n)$ if $n\ge n_0$.
\item[3)] $m_R\in \mbox{Ass}(R/I_{m_0})$ for some $m_0\in \ZZ_{>0}$.
\end{enumerate}
\end{Theorem}

The conclusions of Theorem \ref{Theorem1} do not hold for more general filtrations.

\begin{Example} There exist $\RR$-divisorial filtrations $\mathcal I$ such that 
the conclusions of Theorem \ref{Theorem1} are false. 
\end{Example}

\begin{proof} Let $R=k[[x]]$ be a power series ring in one variable over a field $k$. Let $I_n=(x^{\lceil n\pi\rceil})$ for $n\in \NN$, and $\mathcal I=\{I_n\}$. For fixed $n$, $r\lceil n\pi\rceil>\lceil rn\pi\rceil +1$ for some  $r\in \ZZ_{>0}$. Thus $f\in I_n$ implies $f^r\in m_RI_{nr}$. Thus $\ell(\mathcal I)=\dim R[\mathcal I]/m_RR[\mathcal I]=0<1=\dim R$.
\end{proof}

Theorem \ref{Theorem2} is proven for two dimensional normal excellent local rings in \cite{C4}. In \cite{C4}, the technique of Zariski decomposition of divisors on two dimensional nonsingular schemes is used to reduce the analysis of $\QQ$-divisors on a resolution of singularity of a two dimensional excellent normal domain to the case of numerically effective $\QQ$-divisors, which have very good properties, allowing the proof  of the equivalence  of all of the  statements of Theorem \ref{Theorem2} in dimension 2, using these geometric methods. Zariski decomposition does not exist in higher dimensions, even after blowing up (\cite{C1},  \cite{CS}, \cite[Section IV.2.10]{Nak}, \cite[Section 2.3]{LA}), so a different method must be used in higher dimensions.

 \section{Divisorial filtrations on normal excellent local rings} Let $R$ be a normal excellent local ring. Let $\mathcal I=\{I_m\}$ where 
 $$
I_m=I(\nu_1)_{ma_1}\cap\cdots\cap I(\nu_s)_{ma_s}.
$$ 
 for some divisorial valuations $\nu_1,\ldots,\nu_s$ of $R$ be an $\RR$-divisorial filtration of $R$, with $a_1,\ldots, a_s\in \RR_{>0}$. Then there exists a projective birational morphism $\phi:X\rightarrow \mbox{Spec}(R)$ such that there exist prime divisors $F_1,
 \ldots, F_s$ on $X$ such that $V_{\nu_i}=\mathcal O_{X,F_i}$ for  $1\le i\le s$ (\cite[Remark 6.6 to Lemma 6.5]{CPS1}). Let $D=a_1F_1+\cdots+a_sF_s$, an effective $\RR$-divisor on $X$ (an effective $\RR$-Weil divisor). Define $\lceil D\rceil=\lceil a_1\rceil F_1+\cdots+\lceil a_s\rceil F_s$, an integral divisor. 
 We have coherent sheaves $\mathcal O_X(-\lceil n D\rceil)$ on $X$ such that 
 \begin{equation}\label{N1}
 \Gamma(X,\mathcal O_X(-\lceil nD\rceil ))=I_n
 \end{equation}
 for $n\in \NN$. If $X$ is nonsingular then $\mathcal O_X(-\lceil nD\rceil)$ is invertible. The formula (\ref{N1}) is independent of choice of $X$. Further, even on a particular $X$, there are generally many different choices of effective $\RR$-divisors $G$ on $X$ such that $\Gamma(X,\mathcal O_X(-\lceil nG\rceil))=I_n$ for all $n\in \NN$. Any choice of a divisor $G$ on such an $X$ for which the formula $\Gamma(X,\mathcal O_X(-\lceil nG\rceil))=I_n$ for all $n\in \NN$ holds will be called a representation of the filtration  $\mathcal I$. 
 
 Given an  $\RR$-divisor $D=a_1F_1+\cdots+a_sF_s$ on $X$ we have a divisorial filtration
 $\mathcal I(D)=\{I(nD)\}$ where 
 $$
 I(nD)=\Gamma(X,\mathcal O_X(-\lceil nD\rceil ))=I(\nu_1)_{\lceil na_1\rceil}\cap\cdots\cap I(\nu_s)_{\lceil na_s\rceil} 
 =I(\nu_1)_{na_1}\cap\cdots\cap I(\nu_s)_{na_s}.
 $$ 
 We write $R[D]=R[\mathcal I(D)]$.
 
 The following result will be used in our proof.
 
\begin{Lemma}\label{LemmaG} Suppose that $R$ is a universally Nagata domain and $X$ is an integral, projective $R$-scheme. Suppose that $D$ is an ample Cartier divisor on $X$. Then the algebra $\oplus_{n\ge 0}\Gamma(X,\mathcal O_X(nD))$ is a finitely generated $R$-algebra.
\end{Lemma}

We give an outline of the proof of this well known result. There exists $m\in \ZZ_{>0}$ such that $\mathcal O_X(mD)$ is very ample. By the argument from the last two lines of page 122 of the proof of \cite[Theorem II.5.19]{H} and \cite[Remark[5.19.2]{H}, the algebra
$A=\oplus_{n\ge 0}\Gamma(X,\mathcal O_X(nmD))$ is a finitely generated $R$-algebra.  There exists $a\in \ZZ_{>0}$ such that for $0\le r<m$, 
$\Gamma(X,\mathcal O_X(r+am)D)$ is generated by global sections. Thus there exist short exact sequences
$0\rightarrow \mathcal K_r\rightarrow \mathcal O_X^{b_r}\rightarrow \mathcal O_X((r+am)D)\rightarrow 0$ 
for some $b_r\in \ZZ_{>0}$ and there exists $c\in \ZZ_{>0}$ such that we have surjections $\Gamma(X,\mathcal O_X(nmD))^{b_r}\rightarrow \Gamma(X,\mathcal O_X((nm+r+am)D)$ for all $n\ge c$ and $0\le r<m$. Thus,
since every module $\Gamma(X,\mathcal O_X(nD))$ is a finitely generated $R$-module, we have that $\oplus_{n\ge 0}\Gamma(X,\mathcal O_X((r+nm)D)$ is a finitely generated $A$-module for $0\le r<m$, and so
$\oplus_{n\ge 0}\Gamma(X,\mathcal O_X(nD))$ is a finitely generated $R$-algebra.

\section{Divisors on a resolution of singularities}\label{SecDiv} 
Let $R$ be a normal excellent local ring with maximal ideal $m_R$. Let $K$ be the quotient field of $R$ and let $k=R/m_R$ be its residue field.  Let $\pi:X\rightarrow \mbox{Spec}(R)$ be a birational projective morphism such that $X$ is nonsingular.   

A divisor (or integral divisor) $D$ on $X$ is a sum $D=\sum a_iE_i$ where $a_i$ are integers and $E_i$ are prime divisors on $X$. $D$ is a $\QQ$-divisor if the $a_i\in \QQ$ and $D$ is an $\RR$-divisor if the $a_i$ are in $\RR$. The support of $D$ is the algebraic set $\cup_{a_i\ne 0}E_i$. $D$ is an effective divisor if all $a_i$ are nonnegative. If $D_1$ and $D_2$ are divisors, then $D_2\ge D_1$ if $D_2-D_1$ is an effective divisor. 

If $D_1$ and $D_2$ are $\RR$-divisors on $X$, then $D_1$ and $D_2$ are linearly equivalent, written as $D_1\sim D_2$, if there exists $f\in K$ such that $(f)=D_1-D_2$. Here $(f)=\sum \nu_E(f)E$ where the sum is over all prime divisors $E$ on $X$ and $\nu_E$ is the valuation of the valuation ring $\mathcal O_{X,E}$.

Let $D=\sum a_iE_i$  be an $\RR$-divisor on $X$. There is an associated integral divisor 
$\lfloor D\rfloor = \sum \lfloor a_i\rfloor E_i$ where $\lfloor a\rfloor$ is the round down of a real number $a$.
Analogously, we define $\lceil D\rceil$ to be $\sum \lceil a_i\rceil E_i$ where $\lceil a\rceil$ is the round up of $a$. We have that
$-\lceil D\rceil =\lfloor -D\rfloor$.

Associated to $D$ is an invertible sheaf $\mathcal O_X(\lfloor D\rfloor)$, which is defined by
\begin{equation}\label{eq6}
\mathcal O_X(\lfloor D\rfloor)|U=\frac{1}{\prod f_i^{\lfloor a_i\rfloor}}\mathcal O_X|U
\end{equation}
whenever $U$ is an open subset of $X$ such that the $E_i$ have local equations $f_i=0$ in $\Gamma(U,\mathcal O_X)$.

For $V$ an open subset of $X$,
$$
\Gamma(V,\mathcal O_X(\lfloor D\rfloor))=\left\{f\in K\mid ((f)+\lfloor D\rfloor )\cap V\ge 0\right\}=\left\{f\in K\mid ((f)+D)\cap V\ge 0\right\}
$$
since for $f\in K$, $(f)+D\ge 0$ if and only if $(f)+\lfloor D\rfloor\ge 0$.

If $D_1$ and $D_2$ are $\RR$-divisors, then there is a natural map 
$$
\mathcal O_X(\lfloor D_1\rfloor)\otimes\mathcal O_X(\lfloor D_2\rfloor)\rightarrow \mathcal O_X(\lfloor D_1+D_2\rfloor).
$$
giving a natural map of $R$-modules 
$$
\Gamma(X,\mathcal O_X(\lfloor D_1\rfloor))\otimes\Gamma(X\mathcal O_X(\lfloor D_2\rfloor))\rightarrow \Gamma(X,\mathcal O_X(\lfloor D_1+D_2\rfloor)).
$$

This multiplication is somewhat subtle if $D_1$ and $D_2$ are not integral divisors. With the notation of (\ref{eq6}), suppose that $U$ is affine. Suppose that $D_1=\sum a_iE_i$ and $D_2=\sum b_iE_i$. Then 
$$
\mathcal O_X(\lfloor D_1\rfloor)|U\otimes \mathcal O_X(\lfloor D_2\rfloor)|U\rightarrow \mathcal O_X(\lfloor D_1+D_2\rfloor)|U
$$
is the map
\begin{equation}\label{eq7}
\frac{g}{\prod f_i^{\lfloor a_i\rfloor}}\otimes \frac{h}{\prod f_i^{\lfloor b_i\rfloor}}\mapsto 
\frac{gh\prod f_i^{c_i}}{\prod f_i^{\lfloor a_i+b_i\rfloor}}
\end{equation}
where $g,h\in \mathcal O_X(U)$ and $0\le c_i=\lfloor a_i+b_i\rfloor -(\lfloor a_i\rfloor+\lfloor b_i\rfloor)$.

 Let $D$ be an integral divisor on $X$. Write $D=G-F$ where $G$ and $F$ are effective integral divisors. Then 
 the natural inclusion $\mathcal O_X(-G)\rightarrow \mathcal O_X$ induces an inclusion $\mathcal O_X\rightarrow \mathcal O_X(G)$, and thus an inclusion $\mathcal O_X(-F)\rightarrow \mathcal O_X(D)$. Taking global sections, we have an inclusion 
 $\Gamma(X,\mathcal O_X(-F))\rightarrow \Gamma(X,\mathcal O_X(D))$. Now 
 $\Gamma(X,\mathcal O_X(-F))$ is an intersection of valuation ideals in $R$ so that $\Gamma(X,\mathcal O_X(D))\ne 0$. 
 In particular, we see that there exists an effective integral divisor $G$ on $X$ such that $G\sim D$.

\begin{Lemma}\label{LemmaD} Suppose that $R$ is an excellent local domain such that either $R$ is of equicharacteristic zero or $\dim R\le 3$ and that $\nu_1,\ldots,\nu_r$ are divisorial valuations of $R$. Then there exists a birational projective morphism $\pi:X\rightarrow \mbox{Spec}(R)$ such that $X$ is nonsingular and there exist prime divisors $F_1,\ldots,F_r$ on $X$ such that  $\mathcal O_{X,F_i}$ is the valuation ring of $\nu_i$ for $1\le i\le r$.
\end{Lemma}

 \begin{proof} By \cite[Remark 6.6]{CPS1} to \cite[Lemma 6.5]{CPS1}, there exists a birational projective morphism $Y\rightarrow \mbox{Spec}(R)$ such that $Y$ is normal and there exist prime divisors $E_1,\ldots,E_r$ on $Y$ such that $\mathcal O_{Y,E_i}$ is the valuation ring of $\nu_i$ for $1\le i\le r$. There exists a birational projective morphism $X\rightarrow Y$ such that $X$ is nonsingular by \cite[Main Theorem I(n), page 145]{Hi} or \cite{Te} if $R$ is excellent of equicharacteristic zero, and there exists such a resolution of singularities by \cite{CP} if $\dim R\le 3$.
 \end{proof}

\section{The $\gamma_{\Gamma}$ function}

We recall the $\gamma_{\Gamma}$ function defined in \cite[Section 3]{C5}. The statements and proofs in this section are based on corresponding statements and proofs for the $\sigma_{\Gamma}$ function on pseudo-effective divisors on a projective nonsingular variety in \cite[Chapter III, Section 1]{Nak}. 

We continue to  assume that $R$ is a normal excellent local ring and that $\pi:X\rightarrow \mbox{Spec}(R)$ is a birational projective morphism such that $X$ is nonsingular. 
Let $G=\sum a_iE_i$ be an effective $\RR$-divisor, and $\Gamma$ be a prime divisor on $X$. 
Let
$$
\mbox{ord}_{\Gamma}(G)=\left\{\begin{array}{ll}
a_i&\mbox{ if }\Gamma=E_i\\
0&\mbox{if }\Gamma\not\subset \mbox{Supp}(D).
\end{array}\right.
$$
For $D$ an $\RR$-divisor, let 
$$
\tau_{\Gamma}(D)=\inf\{\mbox{ord}_{\Gamma}(G)\mid G\ge 0\mbox{ and }G\sim D\},
$$
and define
$$
\gamma_{\Gamma}(D)=\inf\left\{\frac{\tau_{\Gamma}(mD)}{m}\mid m\in \ZZ_{>0}\right\}.
$$
Since $D$ is linearly equivalent to an effective divisor, there can be  only finitely many prime divisors $\Gamma$ such that $\gamma_{\Gamma}(D)>0$.

If $R$ has dimension 2 and $D$ is an integral divisor on $X$ then $\gamma_{\Gamma}(D)$ is a rational number, but there exists examples where $R$ has dimension 3 and integral divisors $D$ on $X$ such that $\gamma_{\Gamma}(D)$
 is an irrational number (\cite[Theorem 4.1]{C3}).

We have that $\tau_{\Gamma}(D_1+D_2)\le \tau_{\Gamma}(D_1)+\tau_{\Gamma}(D_2)$ for $\RR$-divisors $D_1$ and $D_2$.
If $f:\NN\rightarrow \RR_{\ge 0}$ is a function such that $f(k_1+k_2)\le f(k_1+k_2)$ for all $k_1,k_2$, then $\lim_{k\rightarrow\infty}\frac{f(k)}{k}$ exists (\cite[Lemma III.1.3]{Nak}). Hence 
$$
\gamma_{\Gamma}(D)=\lim_{m\rightarrow \infty}\frac{\tau_{\Gamma}(mD)}{m}.
$$
Thus if $\alpha\in \QQ_{>0}$, we have that
\begin{equation}\label{eq10}
\gamma_{\Gamma}(\alpha D)=\alpha\gamma_{\Gamma}(D).
\end{equation}
We also have that for $\RR$ divisors $D_1$ and $D_2$,
\begin{equation}\label{eq11}
\gamma_{\Gamma}(D_1+D_2)\le\gamma_{\Gamma}(D_1)+\gamma_{\Gamma}(D_2).
\end{equation}
If $m\in \ZZ_{>0}$ and $G\sim mD$ is an effective $\RR$-divisor, then $\mbox{ord}_{\Gamma}(G)\ge m\gamma_{\Gamma}(D)$.

\begin{Lemma}\label{LemmaE} Let $D$ be an $\RR$-divisor and $\Gamma$ be a prime divisor on $X$. Then 
$$
\Gamma(X,\mathcal O_X(\lfloor nD\rfloor))=\Gamma(X,\mathcal O_X(\lfloor n(D-\gamma_{\Gamma}(D)\Gamma)\rfloor))
$$
for all $n\in \NN$.
\end{Lemma} 

\begin{proof} We have that $\lfloor n(D-\gamma_{\Gamma}(D)\Gamma)\rfloor\le \lfloor nD\rfloor$ for all $n\in \NN$ since $\gamma_{\Gamma}(D)\ge 0$. Thus there is a natural inclusion $\Gamma(X,\mathcal 
O_X(\lfloor n(D-\gamma_{\Gamma}(D)\Gamma)\rfloor)\rightarrow
\Gamma(X,\mathcal O_X(\lfloor nD\rfloor))$ for all $n\in \NN$.

Suppose that $f\in \Gamma(X,\mathcal O_X(\lfloor nD\rfloor))$. Then $(f)+nD\ge 0$ so $\mbox{ord}_{\Gamma}((f)+nD)\ge n\gamma_{\Gamma}(D)$.  Thus $(f)+nD=U+n\gamma_{\Gamma}(D)$ for some effective $\RR$-divisor $U$ which implies that $(f)+\lfloor n(D-\gamma_{\Gamma}(D)\Gamma)\rfloor$ is effective. Thus $\Gamma(X,\mathcal O_X(\lfloor n(D-\gamma_{\Gamma}(D)\Gamma)\rfloor))=\Gamma(X,\mathcal O_X(\lfloor nD\rfloor))$.
\end{proof}

\begin{Lemma}\label{LemmaA} Let $D$ be an $\RR$-divisor and $\Gamma$ be a prime divisor on $X$ such that $\gamma_{\Gamma}(D)>0$. Let $s\in \RR$ be such that  $0\le s\le \gamma_{\Gamma}(D)$. Then $\gamma_{\Gamma}(D-s\Gamma)=\gamma_{\Gamma}(D)-s$.
\end{Lemma}

\begin{proof} Given $\epsilon\in \RR_{>0}$, there exists $m\in \ZZ_{>0}$ and an $\RR$-divisor $G$ such that $G\sim mD$ and
 $\mbox{ord}_{\Gamma}(G)\le m(\gamma_{\Gamma}(D)+\epsilon)$. Since $\mbox{ord}_{\Gamma}(G)\ge m\gamma_{\Gamma}(D)$, $G-ms\Gamma$ is effective, with $\mbox{ord}_{\Gamma}(G-ms\Gamma)\le m(\gamma_{\Gamma}(D)-s+\epsilon)$. Thus $\gamma_{\Gamma}(mD-s\Gamma)\le \gamma_{\Gamma}(D)-s$.
 
 Suppose that $\Lambda\sim m(D-s\Gamma)$ for some $m\in \ZZ_{>0}$ is an effective $\RR$-divisor. Then 
 $\Lambda+ms\Gamma\sim mD$ so that $\mbox{ord}_{\Gamma}(\Lambda+ms\Gamma)\ge m\gamma_{\Gamma}(D)$. Thus
 $\mbox{ord}_{\Gamma}(\Lambda)\ge m(\gamma_{\Gamma}(D)-s)$ and so $\gamma_{\Gamma}(D-s\Gamma)\ge \gamma_{\Gamma}(D)-s$.
 \end{proof}

\begin{Lemma}\label{LemmaB} Let $D$ be an $\RR$-divisor and $\Gamma$ be a prime divisor on $X$ such that $\gamma_{\Gamma}(D)>0$. Then for $s\in \RR_{\ge 0}$, $\gamma_{\Gamma}(D+s\Gamma)=\gamma_{\Gamma}(D)+s$.
\end{Lemma}

\begin{proof} 
Let $E=D+s\Gamma$. Let $\gamma=\gamma_{\Gamma}(D)$. For $0<c<1$ we have 
$$
(1-c)(D-\gamma\Gamma)+cE=D+(-(1-c)\gamma+cs)\Gamma.
$$ 
Let $c$ be a rational number with $0<c<\frac{\gamma}{s+\gamma}$. Then
$-\gamma<-(1-c)\gamma+cs<0$. By (\ref{eq11}),
$$
\gamma(D+(-(1-c)\gamma+cs)\Gamma)\le \gamma_{\Gamma}((1-c)(D-\gamma\Gamma))+\gamma_{\Gamma}(cE).
$$
By Lemma \ref{LemmaA} and (\ref{eq10}),
$$
c(\gamma+s)=\gamma_{\Gamma}(D)+(-(1-c)\gamma+cs)\le(1-c)\gamma_{\Gamma}(D-\gamma\Gamma)+c\gamma_{\Gamma}(E)
=c\gamma_{\Gamma}(E). 
$$
Thus $\gamma_{\Gamma}(E)\ge\gamma+s$. By (\ref{eq11}), we have that 
$$
\gamma_{\Gamma}(E)\le\gamma_{\Gamma}(D)+\gamma_{\Gamma}(s\Gamma)\le \gamma_{\Gamma}(D)+s.
$$	
\end{proof}

\begin{Lemma}\label{LemmaC} Let $\Gamma$ be a prime divisor on $X$ and $D$ be a $\QQ$-divisor on $X$ such that $\gamma_{\Gamma}(D)=0$. Let $A$ be an ample $\QQ$-divisor on $X$. Then there exists an effective integral divisor $B$ and $m>0$ such that   $\Gamma$ is not in the support of $B$, 
$m(D+A)$ is an integral divisor and $B\sim m(D+A)$.
\end{Lemma}

\begin{proof} There exists $n_0\in \ZZ$ such that $n_0A$ is an ample integral divisor and $n_0A+\Gamma$ is an ample integral divisor (\cite[Exercise II.7.5]{H}). There exists $m\in \ZZ_{>0}$ and an integral effective divisor $U$ such that $mD$ and $mA$ are integral divisors, $U\sim mD$ and $\alpha:=\mbox{ord}_{\Gamma}U<\frac{m}{n_0}$. Necessarily, $\alpha\in \NN$. Let $\overline U=U-\alpha\Gamma$. $\overline U$ is an integral divisor which does not have $\Gamma$ in it's support. 
$$
m(D+A)\sim \overline U+\alpha\Gamma+mA=\overline U+\alpha(n_0A+\Gamma)+(m-n_0\alpha)A
$$
with $m-n_0\alpha>0$. Further, $\overline U$, $\alpha(n_0A+\Gamma)$ and $(m-n_0\alpha)A$ are all three integral divisors. 
Since $n_0A+\Gamma$ and $A$ are ample $\QQ$-divisors, there exists $n\in \ZZ_{>0}$ and effective integral divisors $D_1$ and $D_2$ which do not contain $\Gamma$ in their supports such that $D_1\sim n(n_0A+\Gamma)$ and  and $D_2\sim n(m-n_0\alpha)A$. Thus 
$$
n(\overline U+\alpha(n_0A+\Gamma)+(m-n_0\alpha)A)\sim n\overline U+\alpha D_1+D_2
$$
is an effective divisor which does not contain $\Gamma$ in its support.
The divisor 
$$
B=(n\overline U+\alpha D_1+D_2)
$$
satisfies the conclusions of the lemma (with $B\sim nm(D+A)$).
\end{proof}

\section{Analytic spread and associated primes}
 
In this section we prove Theorem \ref{Theorem1} in the case that $R$ is normal. We assume that $R$ is normal throughout this section. Let $k=R/m_R$.
There exists by Lemma \ref{LemmaD}, a projective birational morphism $\pi:X\rightarrow \mbox{Spec}(R)$ such that $X$ is nonsingular 
 and there exists an effective $\QQ$-divisor $D$ on $X$ such that $\mathcal I=\mathcal I(D)$. Let 
 \begin{equation}\label{eq**}
 D=\sum a_iF_i
\end{equation}
with the $F_i$ distinct prime divisors and $a_i\in \QQ_{\ge 0}$.

After reindexing the $F_i$, we may assume that the first $s$ prime divisors $F_1,\ldots,F_s$ are the components of $D$ which contract to $m_R$. For $n\in \ZZ_{>0}$, we have that 
$m_R\in \mbox{Ass}(R/I_n)$ if and only if $I(nD)\ne I(n(\sum_{i>s}a_iF_i))$
which holds  if and only if there exists $j$ with $1\le j\le s$ such that $I(nD)\ne I(n(\sum_{i\ne j}a_iF_i))$.
Since $m_R\in \mbox{Ass}(R/I_{m_0})$ for some $m_0$, we have that some $F_j$ contracts to $m_R$ for some $j$ with 
$1\le j\le s$, $a_j>0$ and 
\begin{equation}\label{eq*}
I(m_0D)\ne I(m_0(\sum_{i\ne j}a_iF_i)).
\end{equation}

Since $F_j\subset \pi^{-1}(m_R)$, $F_j$ is a projective $k$-variety.  Let $L=\Gamma(F_j,\mathcal O_{F_j})$. Taking global sections of  the exact sequence 
$$
0\rightarrow \mathcal O_X(-F_j)\rightarrow \mathcal O_X\rightarrow \mathcal O_{F_j}\rightarrow 0,
$$
and since $\Gamma(X,\mathcal O_X(-F_j))=m_R$,
we obtain that $\Gamma(X,\mathcal O_X)/\Gamma(X,\mathcal O_X(-F_j))=R/m_R=k$ and 
we obtain an inclusion $k=R/m_R\rightarrow L$. Since $\pi$ is a projective morphism, we have that  
$L$ is a finitely generated $R$-module, and since $F_j$ is a projective variety, $L$ is a field. Thus $L$ is a finite extension field of $k$.

 Let $D_1=\sum_{i\ne j}a_iF_i$, so that
$-D\le -D_1=-D+a_jF_j$. Suppose that $\gamma_{F_j}(-D)>0$. Then $\gamma_{F_j}(-D_1)=\gamma_{F_j}(-D)+a_j$ by Lemma \ref{LemmaB}. Thus 
$$
\Gamma(X,\mathcal O_X(\lfloor -nD\rfloor))=\Gamma(X,\mathcal O_X(\lfloor -nD_1\rfloor))
$$
 for all $n\in \NN$ by Lemma \ref{LemmaE}, which is a contradiction to (\ref{eq*}). Thus $\gamma_{F_j}(-D)=0$ and $0\le \gamma_{F_j}(-D_1)<a_j$.

Let $t\in \QQ$ be such that $0<t<a_j-\gamma_{F_j}(-D_1)$. Let $D_2=D-tF_j$. Suppose that $\gamma_{F_j}(-D_2)>0$. Then 
by Lemmas \ref{LemmaA} and  \ref{LemmaB},
$$
\begin{array}{lll}
0=\gamma_{F_j}(-D_1-\gamma_{F_j}(-D_1)F_j)&=&\gamma_{F_j}(-D_2+(a_j-\gamma_{F_j}(-D_1)-t)F_j)\\
&=&\gamma_{F_j}(-D_2)+(a_j-\gamma_{F_j}(-D_1)-t)>0,
\end{array}
$$
 a contradiction. Thus $\gamma_{F_j}(-D_2)=0$.

Since $X$ is the blowup of an ideal $I$ in $R$, there exists an ample integral divisor $A$ on $X$ such that $I\mathcal O_X=\mathcal O_X(A)$. We may assume that all prime divisors in the support of $A$ appear amongst the $F_i$ in the expansion (\ref{eq**}) of $D$ by possibly adding some $F_i$ with $a_i=0$ to (\ref{eq**}).
Expand $A=\sum -b_iF_i$ with $b_i\in \NN$. The support of $A$ must contain all prime divisors on $X$ which contract to $m_R$, so $b_j>0$. We have that $F_j=-\frac{1}{b_j}A-\sum_{i\ne j}\frac{b_i}{b_j}F_i$, and
$$
-D=-D_2-tF_j=-D_2+\frac{t}{b_j}A+t(\sum_{i\ne j}\frac{b_i}{b_j}F_j).
$$
Let $\Lambda=-D_2+t(\sum_{i\ne j}\frac{b_i}{b_j}F_j)$. 
We have that $\gamma_{F_j}(\Lambda)=0$ since $\gamma_{F_j}(-D_2)=0$. Since $\Lambda$ is a $\QQ$-divisor,
there exists $m\in\ZZ_{>0}$ and an effective integral divisor $U$ such that $m(\Lambda+\frac{t}{2b_j}A)$ is an integral divisor, $F_j$ is not in the support of $U$ and 
$m(\Lambda+\frac{t}{2b_j}A)\sim U$ by Lemma \ref{LemmaC}. 

Let $\Theta = \Lambda+\frac{t}{2b_j}A$, so that $-D=\Theta+\frac{t}{2b_j}A$. 
Since $F_j$ is not in the support of $U$, $U$ restricts to 
an effective Cartier divisor $\overline U$ on $F_j$. We have natural isomorphisms 
$$
\mathcal O_X(m\Theta)\otimes\mathcal O_{F_j}\cong \mathcal O_X(U)\otimes \mathcal O_{F_j}\cong \mathcal O_{F_j}(\overline U).
$$
Let $\tau\in \Gamma(F_j,\mathcal O_X(m\Theta)\otimes\mathcal O_{F_j})$ be the (nonzero) section corresponding to $\overline U$, giving an inclusion $\tau:\mathcal O_{F_j}\rightarrow \mathcal O_X(m\Theta)\otimes\mathcal O_{F_j}$.

 Since $\frac{t}{2b_j}A$ is an ample $\QQ$-divisor, after possibly replacing $m$ with a multiple of $m$,  
  we may further assume that $-\frac{mt}{2b_j}A$ is a very ample integral  divisor and by \cite[Theorem II.5.2]{H},
  $$
  H^1(X,\mathcal O_X(n\frac{mt}{2b_j}A-F_j))=0\mbox{ for all $n\ge 1$}.  
  $$
  We have that
 $\mathcal O_X(\frac{mt}{2b_j}A)\otimes\mathcal O_{F_j}$ is a very ample invertible sheaf on the projective $k$-variety $F_j$, so that  $\overline S=\oplus_{n\ge 0}\Gamma(F_j,\mathcal O_X(n\frac{mt}{2b_j}A)\otimes\mathcal O_{F_j})$ is a finitely generated graded $k$-algebra (by Lemma 6), of dimension $\dim F_j+1=\dim R$.  Let $\overline S_{>0}=\oplus_{n>0}\Gamma(F_j,\mathcal O_X(n\frac{mt}{2b_j}A)\otimes\mathcal O_{F_j})$, the graded maximal ideal of $\overline S$.  The ring
 $\oplus_{n\ge 0}\Gamma(X,\mathcal O_X(n\frac{mt}{b_j}A))$ is a finitely generated graded $R$-algebra (by Lemma \ref{LemmaG}), and so its image $k\oplus \overline S_{>0}$ in $S$ is a finitely generated graded $k$-algebra. 
 Since $L$ is a finite extension field of $k$, $\overline S$ is finite over $S$ and so $\dim S=\dim R$ (by \cite[Corollary A.8]{BH}).
 
  The section $\tau$ induces inclusions
 $$
 \mathcal O_X(n\frac{mt}{2b_j}A)\otimes\mathcal O_{F_j}\stackrel{\tau^n\otimes{\rm id}}{\rightarrow}
 \left(\mathcal O_X(nm\Theta)\otimes_{F_j}\right)\otimes\left(\mathcal O_X(n\frac{mt}{2b_j}A)\otimes\mathcal O_{F_j}\right)
 \cong\mathcal O_X(-nmD)\otimes\mathcal O_{F_j}
 $$
 for all $n\in \NN$ (by our choice of $m$, $mD$ is an integral divisor).
 
 For $n\in \NN$, let $\Omega_n$ be the image of the natural surjection 
 $$
 \Gamma(X,\mathcal O_X(\lfloor-nD\rfloor))\rightarrow 
 \Gamma(F_j,\mathcal O_X(\lfloor -nD\rfloor)\otimes \mathcal O_{F_j}).
 $$

 The kernel of this surjection is $\Gamma(X,\mathcal O_X(\lfloor -nD\rfloor)\otimes\mathcal O_X(-F_j))$ 
 which contains 
 $$
 m_R\Gamma(X,\mathcal O_X(\lfloor -nD\rfloor)).
 $$
  Thus  $\Omega_n$ is a $k=R/m_R$-module. Further,
 $\Omega_0=R/m_R$.

  Identifying $S$ with the isomorphic graded $k$-algebra
 $$
 k\oplus\left(\oplus_{n> 0}\Gamma(F_j,\mathcal O_X(n\frac{mt}{2b_j}A)\otimes\mathcal O_{F_j})\tau^n\right),
 $$
  we have that $S$ is a finitely generated graded $k$-subalgebra of the graded $k$-algebra $\oplus_{n\ge 0}\Omega_{nm}$. There exists $\beta\in\ZZ_{>0}$ such that 
 $\mathcal O_X(-mD+\beta\frac{mt}{2b_j}A))$ is ample (\cite[Exercise II.7.5]{H})and $\Gamma(F_j,\mathcal O_X(\beta\frac{mt}{2b_j}A)\otimes\mathcal O_{F_j})\ne 0$. Let
 $$
 T=\oplus_{n\ge 0}\Gamma(F_j,\mathcal O_X(-nmD+n\beta\frac{mt}{2b_j}A)\otimes\mathcal O_{F_j})). 
 $$
 Then $T$ is a finitely generated $k$-algebra such that $\dim T=\dim R$ (by Lemma \ref{LemmaG}). A nonzero section $\lambda$ of $\Gamma(F_j,\mathcal O_X(\beta\frac{mt}{2b_j}A)\otimes\mathcal O_{F_j})$ induces an inclusion 
 $$
 \lambda:\mathcal O_{F_j}\rightarrow \mathcal O_X(\frac{\beta mt}{2b_j}A)\otimes\mathcal O_{F_j},
 $$
  and hence an inclusion
 of graded $k$-algebras $\oplus_{n\ge 0}\Omega_{nm}\subset T$.
  
 A $k$-algebra is called subfinite it is is a subalgebra of a finitely generated $k$-algebra. Thus
$$
S\subset \oplus_{n\ge 0}\Omega_{nm}\subset T
$$
are subfinite, and so
$$
\dim R=\dim S\le \dim \oplus_{n\ge 0}\Omega_{nm}\le \dim T=\dim R
$$
by \cite[Corollary 4.7]{KT}.  The graded extension $\oplus_{n\ge 0}\Omega_{nm}\rightarrow \oplus_{n\ge 0}\Omega_n$
is integral, so 
$$
\dim \oplus_{n\ge 0}\Omega_n\ge \dim \oplus_{n\ge 0}\Omega_{nm}=\dim R
$$
by the going up theorem (\cite[Theorem 5.11]{AM}). Now 
$$
m_R\Gamma(X,\mathcal O_X(\lfloor -lD\rfloor))\subset \Gamma(X,\mathcal O_X(\lfloor -lD\rfloor-F_j))
$$
for all $l$, so there is natural surjection of graded rings
$$
R[D]/m_RR[D]\rightarrow \oplus_{n\ge 0}\Omega_n.
$$
Thus the analytic spread $\ell(\mathcal I)=\dim R[D]/m_RR[D]\ge \dim R$.
Since the analytic spread is bounded above by $\dim R$ (\cite[Lemma 3.6]{CPS}) we have that $\ell(\mathcal I)=\dim R$.

\section{Proof for excellent domains}
In this section we finish the proof of  Theorem \ref{Theorem1}. That is, we do not assume that $R$  is normal.
   
   There exist $m_R$-valuations 
   $\nu_1,\ldots,\nu_t$  and $a_1,\ldots, a_t\in \QQ_{>0}$ such that $\mathcal I=\{I_n\}$ where 
   $I_n=I(\nu_1)_{a_1n}\cap \cdots \cap I(\nu_t)_{a_tn}$ for $n\ge 0$, with
   $I(\nu_i)_{m}=\{f\in R\mid \nu_i(f)\ge m\}$.
   
   Let $S$ be the normalization of $R$ in the quotient field of $R$. Let $\mathfrak m_1,\ldots,\mathfrak m_u$ be the maximal ideals of $S$.
   Let $J(\nu_i)_m=\{f\in S\mid \nu_i(f)\ge m\}$. 
   For $n\in \NN$, let 
   $$
   J_n=J(\nu_1)_{a_1n}\cap \cdots \cap J(\nu_t)_{a_tn}
   $$
   so that $J_n\cap R=I_n$.
   
   Since we assume that $m_R\in \mbox{Ass}_R(R/I_{m_0})$, there exists $y\in R/I_{n_0}$ such that $m_R=\mbox{ann}_R(y)$. We have that $y\in S/J_n$ is nonzero since  $R/I_{m_0}\rightarrow S/J_{m_0}$ is an inclusion. Thus $\mbox{Ann}_S(y)\ne S$. Since maximal elements in the set of annihilators of elements of $S/J_{m_0}$ are prime ideals (by \cite[Theorem 6.1]{Ma}), there exists a prime ideal $Q$ in $S$ which contains $\mbox{Ann}_S(y)$ and is the annihilator of an element $z$ of $S/J_{m_0}$. We have that $Q$ contains $m_RS$ and 
   \begin{equation}\label{eqZ3}
   \sqrt{m_RS}=\cap \mathfrak m_i
   \end{equation}
    so $Q$ is a maximal ideal $\mathfrak m_i$ of $S$. Thus
   \begin{equation}\label{eqZ2}
   \ell(\{(J_n)_{\mathfrak m_i}\})=\dim (S_{\mathfrak m_i})=\dim R,
   \end{equation}
    since Theorem \ref{Theorem1} has been proven for normal excellent local rings. 
  Let $B=\oplus_{n\ge 0}J_n$, which is a graded ring. 
 We thus have by (\ref{eqZ1}), (\ref{eqZ2}) and (\ref{eqZ3})  that $\dim B/m_RB=\dim R$.

  Thus there is a chain of distinct prime ideals 
 $$
 C_0\subset C_1\subset C_2\subset \cdots\subset C_d
 $$
 in $B$ which contain $m_RB$, where $d=\dim R$.

 There is a natural inclusion of graded rings $R[\mathcal I]=\oplus_{n\ge 0}I_n\subset B=\oplus_{n\ge 0}J_n$. We will now show that $B$ is integral over $R[\mathcal I]$. For $a\in \ZZ_{>0}$, let $R[\mathcal I]_a$ be the $a$-th truncation of $R[\mathcal I]$ and $B_a$ be the $a$-th truncation of $B$, so that $R[\mathcal I]_a$ is the subalgebra of $R[\mathcal I]$ generated by $\oplus_{n\le a}I_n$ and $B_a$ is the subalgebra of $B$ generated by $\oplus_{n\le a}J_n$. It suffices to show that homogeneous elements of $B$ are integral over $R[\mathcal I]$.  Suppose that $f\in J_a$ for some $a$.
 Then $f\in B_a$. Let $0\ne x$ be in the conductor of $S$ over $R$. Then $xJ_n\subset I_n$ for all $n$ since $I_n=J_n\cap R$.
Thus $xB_a\subset R[\mathcal I]_a$, so $f^i\in \frac{1}{x}R[\mathcal I]_a$ for all $i\in \NN$, and so the algebra 
$R[\mathcal I]_a[f]\subset \frac{1}{x}R[\mathcal I]_a$. Since $\frac{1}{x}R[\mathcal I]_a$ is a finitely generated $R[\mathcal I]_a$-module and $R[\mathcal I]_a$ is a Noetherian ring, the ring $R[\mathcal I]_a[f]$ is a finitely generated $R[\mathcal I]_a$-module, so that $f$ is integral over $R[\mathcal I]_a$.

We have a chain of prime ideals 
$$
Q_0\subset Q_1\subset Q_2\subset \cdots\subset Q_d
$$
in $R[\mathcal I]$ where $Q_i:=C_i\cap R[\mathcal I]$. The $Q_i$ are all distinct since the $C_i$ are all distinct and $B$ is integral over $R[\mathcal I]$ (by \cite[Theorem A.6 (b)]{BH}). We have  that $m_RR[\mathcal I]\subset Q_0$, so that 
$\dim R[\mathcal I]/m_RR[\mathcal I]\ge d$. Since this is the maximum possible dimension of $R[\mathcal I]/m_RR[\mathcal I]$ by (\ref{eqZ1}), we have that $\ell(\mathcal I)=d$.

 \end{document}